\documentclass[12pt,reqno]{amsart}

\usepackage{amssymb,array}
\usepackage{amsmath}
\usepackage{amsthm}
\usepackage{amsbsy,mathrsfs}
\usepackage{bm}
\usepackage[english]{babel}
\usepackage{graphicx}
\usepackage{color}
\usepackage{hyperref}
\usepackage[toc,page]{appendix}
\usepackage{thmtools}
\usepackage{thm-restate}
\usepackage{bbm}
\usepackage[shortlabels]{enumitem}
\usepackage{cleveref}

\divide
\oddsidemargin by 2 
\theoremstyle{plain}
\newtheorem{theorem}{Theorem}[section]
\newtheorem*{theorem*}{Theorem}
\newtheorem*{conj*}{Conjecture}
\newtheorem{lemma}[theorem]{Lemma}

\newtheorem{corollary}[theorem]{Corollary}

\newtheorem{computational theorem}[theorem]{Computational Theorem}

\theoremstyle{definition}

\newtheorem{remark}[theorem]{Remark}

\newtheoremstyle{named}{}{}{\itshape}{}{\bfseries}{.}{.5em}{\thmnote{#3 }#1}
\theoremstyle{named}
\newtheorem*{namedtheorem}{Characterization of simple curves}

\newcommand{\Sp}{S}
\newcommand{\St}{T}
\newcommand{\emm}{\mathrm{t}_{X,Y}}
\newcommand{\gp}{\gamma_{P'}(x,y)} 
\newcommand{\gq}{\gamma_{Q'}(x,y)} 
\newcommand{\Gp}{\Gamma_{P}(x,y)} 
\newcommand{\Gq}{\Gamma_{Q}(x,y)} 

\newcommand{\U}{\mathcal{U}}
\newcommand{\V}{\mathcal{V}}

\newcommand{\K}{\mathbb{K}}

\newcommand{\Z}{\mathbb{Z}}

\renewcommand{\H}{{\mathbb{H}}}

\newcommand{\la}{\langle}
\newcommand{\ra}{\rangle}

\theoremstyle{plain} 
\newcommand{\thistheoremname}{}
\newtheorem*{genericthm*}{\thistheoremname}

\begin{document}
\title[Explicit Lie bracket of closed geodesics]{Explicit Lie bracket of closed geodesics on a hyperbolic surface with applications}

\author{Moira Chas}
\address{Stony Brook Mathematics Department and Institute for Mathematical Sciences}
\email{moira.chas@stonybrook.edu}

\author{ Arpan Kabiraj}
\address{Department of Mathematics, Indian Institute of Technology Palakkad}
\email{arpaninto@iitpkd.ac.in}
\begin{abstract}
 In this note we develop a tool box of non-Euclidean plane geometry methods that yield a constructive way to define in terms of closed geodesics  the Goldman bracket on deformation classes of closed, directed curves. We use this construction  to algebraically characterize   closed geodesics without self- intersection on  hyperbolic surfaces.

\end{abstract}
\maketitle


\section{Introduction}

In the mid eighties  Goldman \cite{goldman_invariant_1986} defined an unanticipated Lie bracket on the vector space with basis the set of deformation classes of closed, oriented curves on a closed surface. The definition combined two basic operations on curves: the intersection and the loop product. More precisely, the bracket of (the classes of) two curves intersecting transversally is the signed sum of  (the classes of) the  loop products over  intersection points.

Once there is an algebraic structure associated to  geometric objects, mathematicians, often wonder  how much of the geometry  can be recovered from the algebra. In our case, this compelled us to find a way to characterize (classes of) simple closed curves, in terms of the Goldman Lie bracket, leading  to  the following result:
\begin{namedtheorem}  A nonpower conjugacy class $x$ contains a simple representative if and only if one of the following holds:
\begin{enumerate}
    \item $[x,x^n]=0$ for  some $n$ in $\{2,3,\dots\}$ or, 
    \item $[x,\bar{x}]=0$ (where $\bar{x}$ denotes the class of $x$ with opposite orientation)
\end{enumerate}

\end{namedtheorem} \label{Main Theorem}
\begin{remark}
   Part $(2)$ of the above result is recently proved in \cite{alonso2022lie}.  The proof uses a generalization of the combinatorial methods developed by the first author in \cite{chas2004combinatorial}.
\end{remark}

Historical note: The item (2) in the above theorem was unanticipated even by Goldman and was forced on our notice by a large amount of computer evidence. Thus, it can be regarded as the innovative result of this work. Moreover,  its proof seems to require these methods of plane hyperbolic geometry. The quadratic operation $x \longrightarrow [x, \bar{x}]$  now seems very interesting. Item (1) was proven earlier for $n \ge 3$ and for surfaces with boundary by group theory methods (references and more details below).

Since each free homotopy class of closed, directed curves on a hyperbolic surface is represented by a unique closed, directed geodesic, the Goldman Lie bracket of two such closed geodesics is the signed sum of  closed directed geodesics. For each intersection point of two closed geodesics we give an elementary  geometric construction of a geodesic on the universal cover that projects to  the closed geodesic in the class of that term. 

Given two closed directed geodesics on a surface intersecting at a point $P$, a piecewise, directed, closed geodesic is determined by going around one of the geodesics, and then the other, starting and ending at $P$ in both cases. (This piecewise geodesic is, of course, in the class of term of the bracket associated to the point $P$.) The lift of this piecewise geodesic to the universal cover of the surface is piecewise geodesic that zig-zags over a lift of the closed geodesic in the term (because both have the same endpoints on the circle at infinity on the hyperbolic plane). When we have two canceling terms, we have two piecewise geodesics zig-zagging over a lift of the closed geodesic in the term. The possible configurations of these three curves display geodesic triangles and quadrilaterals in the universal cover. By studying these shapes we obtain the main results of the present work.

In a subsequent work the Turaev Lie cobracket will be treated with these types of direct geometric constructions. 


\subsection{History of the Goldman bracket}
 
 This Lie algebra was related to the symplectic structure on the moduli space of homomorphims  of the fundamental group to various Lie groups· Turaev later augmented the discussion with a Lie bialgebra on the space generated by closed curves classes and raised questions about the cobracket and self-intersection of closed curves. One of the authors in studying Turaev's question  about the cobracket extended  the Lie bialgebra structure to the equivariant homology of the  free loop space of oriented manifolds of all dimensions \cite{chas1999string}.  The motivation for this extension of Goldman-Turaev to all manifolds was the possible connection to the Jaco-Stallings equivalent statement to the  Poincar\'e conjecture: ``Any surjective homomorphism from the fundamental group of a closed surface of genus $g$  to $F_g \times F_g$  contains an essential simple closed curve in the kernel.'' (Here, $F_g$ denotes the free group on $g$ generators). Of course, now that the Poincar\'e  conjecture is resolved, Jaco-Stallings is true. So a new problem arises:  why is it true and how can we understand it? This is one compelling reason  to study closed curves on surfaces directly and concretely as we do here.

\subsection{Relation between the Goldman Turaev Lie bialgebra and the intersection and self-intersection of curves}

In \cite{goldman_invariant_1986},  (the same paper where the definition of the Goldman bracket on (classes of) directed closed curves is introduced)  Goldman   proved that if the bracket of two (classes of) curves is zero, and one of the classes contains a  simple curve, then the two classes have disjoint representatives. This result was generalized by one of the authors in \cite{chas2010minimal}: 
the bracket of two classes  counts intersection if one of the classes  is simple.

Turaev defined the cobracket and Lie bialgebra in \cite{turaev1991skein} and remarked that this cobracket of the power of (the class of) a simple curve is zero and wondered whether the converse held.

In \cite{chas2004combinatorial}, the first author gave an algorithmic presentation of the Goldman-Turaev Lie algebra for surfaces with boundary. She wrote a computer program implementing  this presentation  \cite{minh}. Running this program, she found  examples of (classes of) curves on surfaces of genus larger than zero with cobracket zero which are not simple, nor powers of simple curves. Thus, Turaev's question is answered negatively for  surfaces of genus larger than zero.  
Despite extensive computer experiments, Chas was not able to find such counterexamples on surfaces of genus zero, which lead  to the conjecture that Turaev's cobracket detects self-intersection number in genus zero surfaces. 

 In the proof of a  lemma in \cite{goldman_invariant_1986}, it was stated that $[x,\bar{x}]=0$ for all curves $x$. By running the program \cite{minh}, one of the authors found that the opposite statement seems to hold, namely, the number of terms (counted with multiplicity) of $[x,\bar{x}]$ is twice the self-intersection number of $x$. The operation $x \longrightarrow [x,\bar{x}]$ is interesting and still mysterious. At first, one thinks, as Goldman did, that it is always zero. But computer experiments suggest, as we said the opposite. This operation is, in some sense, a quadratic form. Moreover, each of the terms of $[x,\bar{x}]$ is the conjugacy class of a commutator. 

Le Donne \cite{ledonne2008lie}  showed that in genus zero, the cobracket of a class  being equal to zero implies the class contains a power of a simple curve. The conjecture whether  in  genus zero the cobracket counts minimal self-intersection remains open.

Via computer experiments, Chas  explored various ways that the Goldman-Turaev Lie bialgebra might detect the intersection and self-intersection of closed curves on surfaces. Computer experiments suggested the statement that for each non-power class $x$, the number of terms of  $[x,x^n],$ for $n \ge 2$  or $n=-1$, (counted with multiplicity) is the product of $|2n|$ and the  self-intersection number of $x$. 
We  prove here a special case of this conjecture: if $[x,x^n]=0$ $n \ge 2$  or $n=-1$ then the non-power $x$ is  a simple curve.

This result was proven in  \cite{chas2010algebraic} for surfaces with boundary and $n \ge 3$. Similarly, the statement that  $\delta(x^n)$, $n \ge 2$ counts self-intersection was proven for the case of surfaces with boundary and $n \ge 3$ in \cite{chas2015algebraic} (where $\delta$ denotes Turaev's cobracket)
 
In \cite{chas2016extended} it is proven that$[x^p,y^q]$ counts intersection  and self-intersection if $p, q$   even for orientable orbifolds, if $p$ and $q$ are larger than a constant depending on $x$ and $y$.

In \cite{etingof2006casimirs} it is proven that the center of the Goldman Lie algebra on closed surfaces is  generated by the trivial loop. In \cite{kabiraj2016center} it is proven that the center of the Goldman Lie algebra is generated by the trivial loop and  powers of loops parallel to boundary components in all surfaces.  

In Goldman's original work \cite{goldman_invariant_1986} there is also a Lie algebra structure on classes of undirected curves.  In \cite{chas2022lie}   it is proven that the center of this  Lie algebra of undirected curves is  generated by the class of the trivial loop and the classes of loops parallel to boundary components or punctures. Also it was conjectured (suggested by computer experiments) that if the bracket of two undirected curves is zero then these classes have  disjoint  representatives.

In  \cite{cahn2013intersections} it is shown that  the Andersen-Mattes-Reshetikhin bracket accurately counts intersection of two curves. (The Andersen-Mattes-Reshetikhin bracket  is a Poisson bracket defined on the vector space spanned by chord diagrams on a surface.  It can be viewed as a generalization of the Goldman bracket. )

In \cite{cahn2013generalization} a generalization of Turaev's cobracket is given (following the  Andersen-Mattes-Reshetikhin Lie algebra), and it is  proven that this cobracket counts self-intersection. 

See works of  Kawazumi and   Kuno for interesting geometric results about the Goldman Lie algbra expressed more algebraically \cite{kawazumi2012center, kawazumi2014logarithms,kawazumi2015}. 

\section*{Acknowledgements}
We would like to thank Dennis Sullivan for useful and enlightening conversations. The second author acknowledges the support from DST SERB Grant No.: SRG/2022/001210 and  DST SERB Grant No.: MTR/2022/000327. 

\section{Goldman Lie algebra}
Let $\Sigma$ be an oriented surface that admits a complete hyperbolic metric of finite area with (may be empty) geodesic boundary.   

There is a one-to-one correspondence between free homotopy classes of directed closed curves in $\Sigma$ and the set of all conjugacy classes in the fundamental group $\pi_1(\Sigma)$ of $\Sigma,$ both of which will be denoted by $\pi$.  We denote the free homotopy class of a directed  closed curve $x$ by $\la x\ra$. When two curves $x$ and $y$ intersect  in transversal double points, we denote it by $x\pitchfork y$.

Given two free homotopy classes of directed closed curves $\la x\ra $ and $\la y\ra$, choose directed representatives $x$ and $y$ respectively, such that $x\pitchfork y$. The \emph{Goldman Lie bracket of $\la x\ra$ and $\la y\ra$} is defined to be $$[\la x\ra,\la y\ra]=\sum_{P\in x\cap y}\epsilon_P\la x*_Py\ra$$ where $x\cap y$ denotes the set of all intersection points,  
$\epsilon_P$ denotes the sign of the intersection between $x$ and $y$ at $P$, $x*_Py$ denotes the loop product of ${x}$ and ${y}$ at $P$.

Given any ring $\K$, extend the bracket linearly to the free module $\K\pi$ generated by $\pi$ over $\K$. Goldman \cite{goldman_invariant_1986} proved that this bracket is well defined, skew-symmetric and satisfies the Jacobi identity on $\K\pi$. In other words, $\K\pi$  is a Lie algebra, called the \textit{the Goldman Lie algebra.} For convenience of notation we denote $[\la x \ra,\la y\ra]$ simply by $[x,y]$.

Given two closed curves $x$ and $y$, define the \emph{geometric intersection number} between $x$ and $y$ as the smallest number of crossings of representatives of $x$ and $y$ that intersect only in transversal double points. In symbols, 
$$i(x,y)=min\{x\cap y:x\in\la x\ra, y\in\la y\ra,  \text{where  $x\pitchfork y$ }.\}$$ 

If $x$ is a non-power closed curve, define the \emph{self-intersection number of $\la x\ra$} to be the smallest number of crossings of two different representatives of $\la x\ra$ that intersect only in transversal double points. In symbols, 
$$SL(x)=\frac{1}{2}min\{x_1\cap x_2:x_1,x_2\in\la x\ra,  \text{where $x_1\pitchfork x_2$ }\}$$

A curve $x$ is called \emph{simple} if it is not a power of another curve and  $SL(x)$ is zero. In this case we call the class $\la x\ra$ simple.  It is clear from the definition of Goldman bracket that if $\la x\ra$ has a simple representative then $[x^n,x]=0$ for all $n\in\Z$.

\section{Loop product of closed geodesics}
Consider a surface $\Sigma$  of negative Euler characteristic, with  a complete  hyperbolic metric. The length of a closed geodesic $x$ on $\Sigma$ is denoted by $\ell_x$. We identify the universal cover of $\Sigma$ with the   upper half plane $\H$

Given a geodesic $x$ on $\Sigma$, an \textit{$x$-segment} is a geodesic segment of a lift of $x$.  Geodesics and points  will be denoted by lowercase letters ($x,y,\dots$ ) and uppercase letters ($P, Q, \dots$) respectively. 

\subsection{The ``raw product" of the two intersecting closed geodesics}

If $x$ and $y$ are two closed directed geodesics in $\Sigma$ intersecting at a point $P$, denote  by $\Gp$  the  directed, closed, piecewise geodesic on $\Sigma$ that starts at $P$, travels once around $x$,  and then goes once around $y$. (This piecewise geodesic curve is an actual geodesic except for the two turns at $P$, both of the same directed angle. The direction of the turns is determined by the directions of the geodesics, see  Figure~\ref{fig:piecewise}, right).
\begin{figure}[ht]
	\centering
    \includegraphics[width=0.6\textwidth]{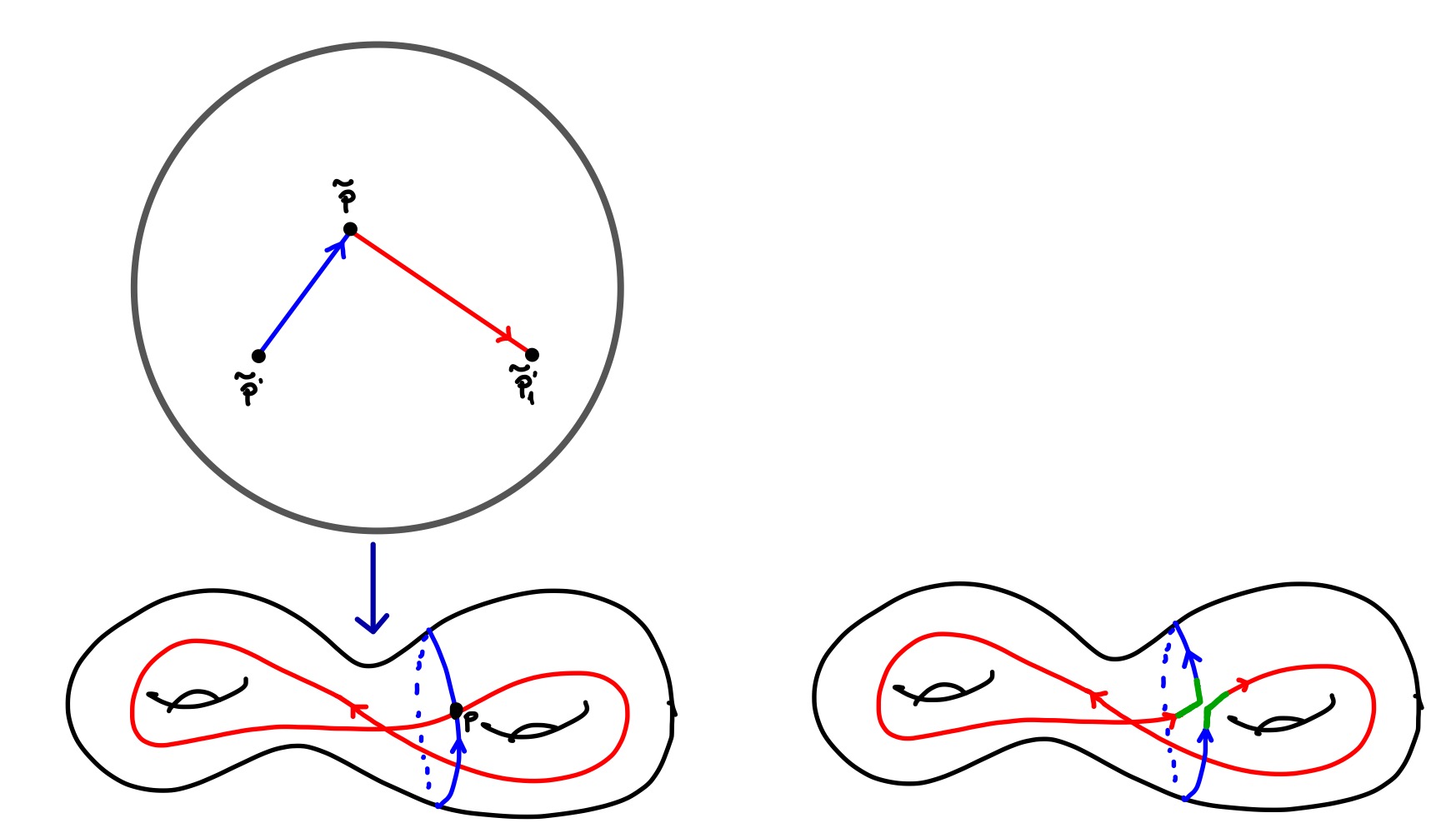}

	\caption{Left: One lap of a lift of $\Gp$ to the universal cover of $\Sigma$. Right: $\Gp$ with ``turning" marked in green. }.
	\label{fig:piecewise}
\end{figure}
\subsection{The piecewise geodesic on the class of the term and its lift, the zigzag curve}
Consider a lift $P'$ of the point  $P$ to the universal cover of $\Sigma$ and a full lift $\gp$  of $\Gp$ as  in Figure~\ref{fig:piecewise2}. 

 Observe that there are two geodesic segments emanating from each preimage of $P$ in $\gp$:  one which is a lap of a lift of $x$ and the other, a  lap of a lift of $y$.  Moreover, either  the $x$-segment arrives to each preimage of $P$ and the $y$ segment leaves or  the other way around. 

\begin{figure}[ht]
	\centering
	\includegraphics[width=1\textwidth]{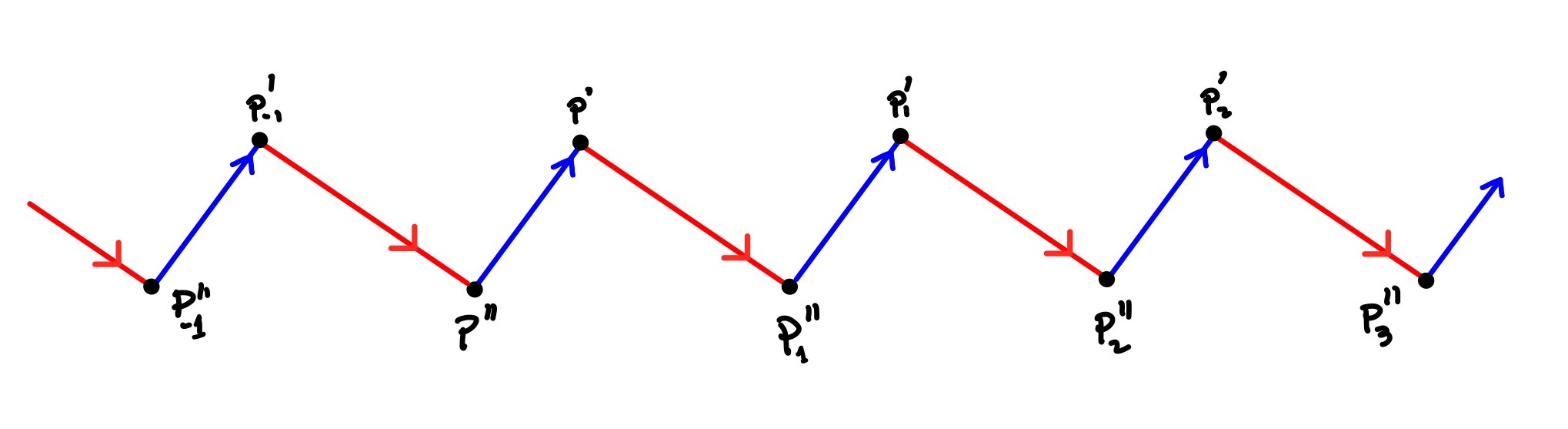}
	\caption{A portion of $\gp$. (The lifts of $x$ are are in blue, and the lifts of $y$ in red.}\label{fig:piecewise2}
\end{figure}

By possibly changing labels, one can assume that  there is an $x$-segment $P''P'$ arriving to $P'$ and a $y$-segment $P'P_1''$ leaving from $P'$ as in Figure~\ref{fig:piecewise2}.

Denote by $\Sp$ the midpoint of  $P''P'$ and by $\St$, the midpoint of $P'P_1''$  (see Figure~\ref{fig:3zigzag}). Note that the lengths of the segments $P''\Sp$ and $\Sp P'$  are each $\ell_x/2$,  (half of the length of the closed geodesic $x$). Similarly, the lengths of the segments   $P'\St$ and   $\St P''_1$ are each $\ell_y/2$.

\begin{figure}[ht]
	\centering
	\includegraphics[width=0.8\textwidth]{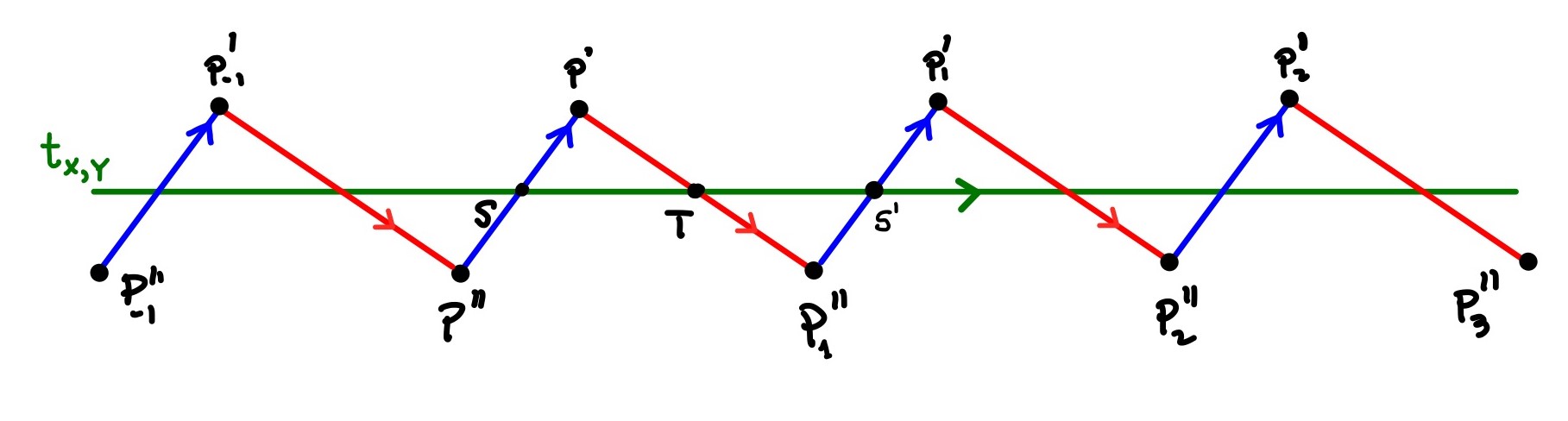}
	\caption{A portion of $\gp$ and the line $\emm$}
	\label{fig:3zigzag}
\end{figure}

Denote by $X$ (resp. $Y$) the hyperbolic transformation on the hyperbolic plane associated to   $x$ (resp. $y$) as elements of the fundamental group of $\Sigma$ based at $P$. Observe the axes of $X$ and $Y$ contain an $x$-segment and an $y$-segment respectively. 
Note  that $X\circ Y$ and $Y\circ X$ are conjugate. Hence, they  have have the same translation length and  the projection of their axes to the surface coincide. 

\begin{lemma}\label{lem:beardon} The following statements hold:
\begin{enumerate}
\item Both compositions $X\circ Y$ and $Y\circ X$ of the transformations $X$ and $Y$ are  hyperbolic.
\item The axis of $X \circ Y$ is the line $\emm$  through $\Sp$ and  $\St$, directed from $\Sp$ to $\St$. This line is a (full) lift of the closed geodesic in the free homotopy class of the closed piecewise geodesic $\Gp$ (see Figure~\ref{fig:3zigzag}). 
\item The   translation length $\ell_{xy}$ of $X \circ Y$ is the double of the distance   between $\Sp$ and $\St$ and is given by the formula $$\cosh{(l_{xy}/2)}=\cosh{(l_{x}/2)}\cosh{(l_{y}/2)}-\sinh{(l_x/2)}\sinh{(l_y/2)}\cos{\alpha},$$ where  $\alpha$ is the angle indicated in Figure~\ref{fig:beardon}.

\end{enumerate}
\end{lemma}
\begin{proof} 
 The proof of 1. and 2. follow  from  \cite[Theorem 7.38.6]{beardon2012geometry} and we include it for completeness.
For each point $U$ on the hyperbolic plane, we denote by $R_{U}$ the anticlockwise rotation  about $U$ of angle $\pi$. 
\begin{figure}[ht]
	\centering
	\includegraphics[width=0.35\textwidth]{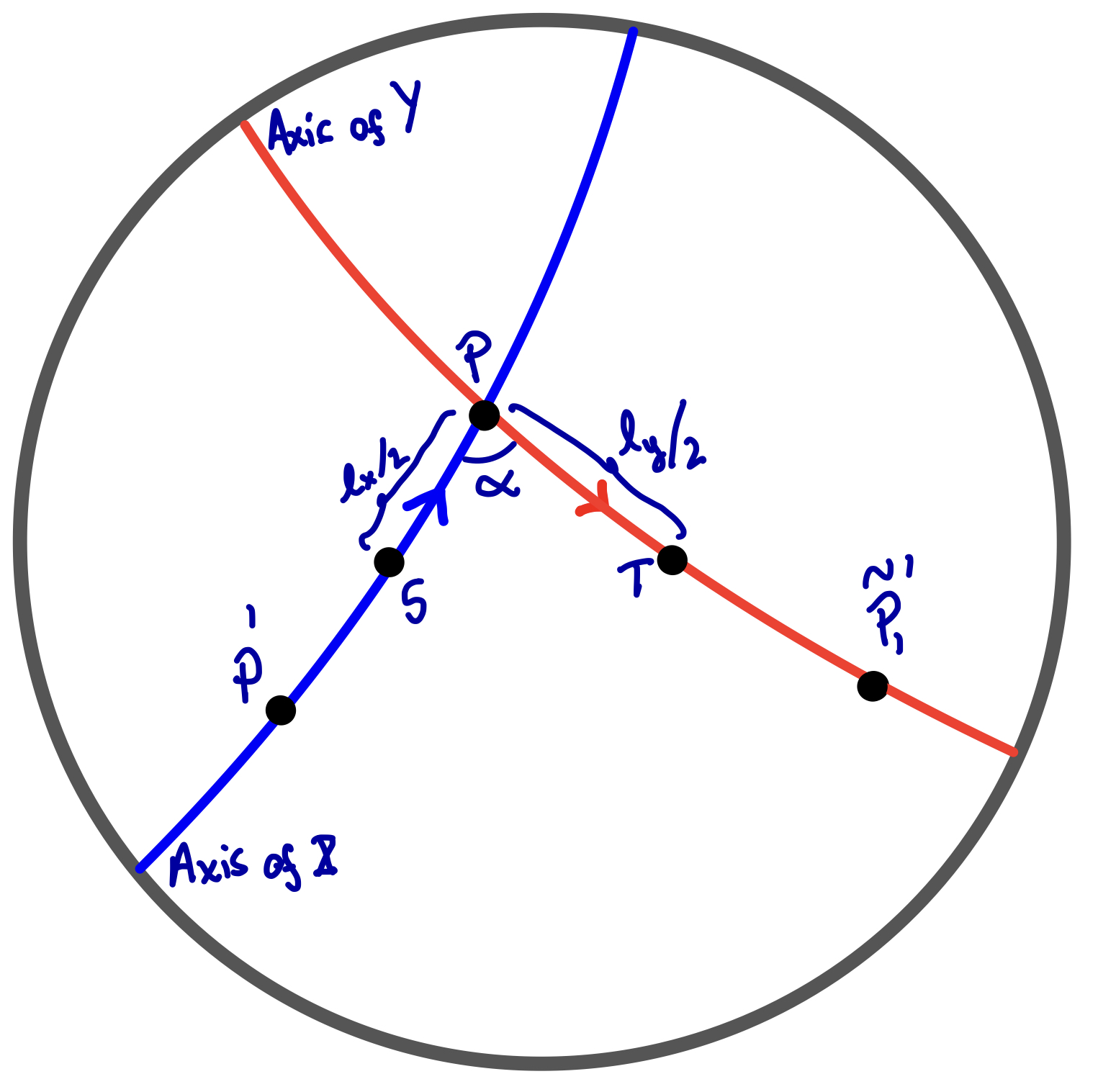}
	\caption{Product of two hyperbolic transformations whose axes intersect}
	\label{fig:beardon}
\end{figure} 
 
Observe that the axis of $X$ contains the segment $\Sp'\Sp$ (in blue in Figure~\ref{fig:beardon}) and the axis of $Y$ contains the segment $\Sp\Sp_1$ (in red in Figure~\ref{fig:beardon}). Also, $X=R_{P} R_{\Sp}$ and $Y=R_{T} R_{P}$. Therefore, $Y \circ X = R_{S} R_{T}$ and the axis of  $Y \circ X$ is the line from $\Sp$ to $T$. By the Cosine Rule (see \cite[Section 1.12]{beardon2012geometry}. Since $X \circ Y$ and $Y \circ X$ are conjugate by $X$, if one of the compositions is hyperbolic so is the other.

\end{proof}

\begin{figure}[ht]
	\centering
	\includegraphics[width=0.8\textwidth]{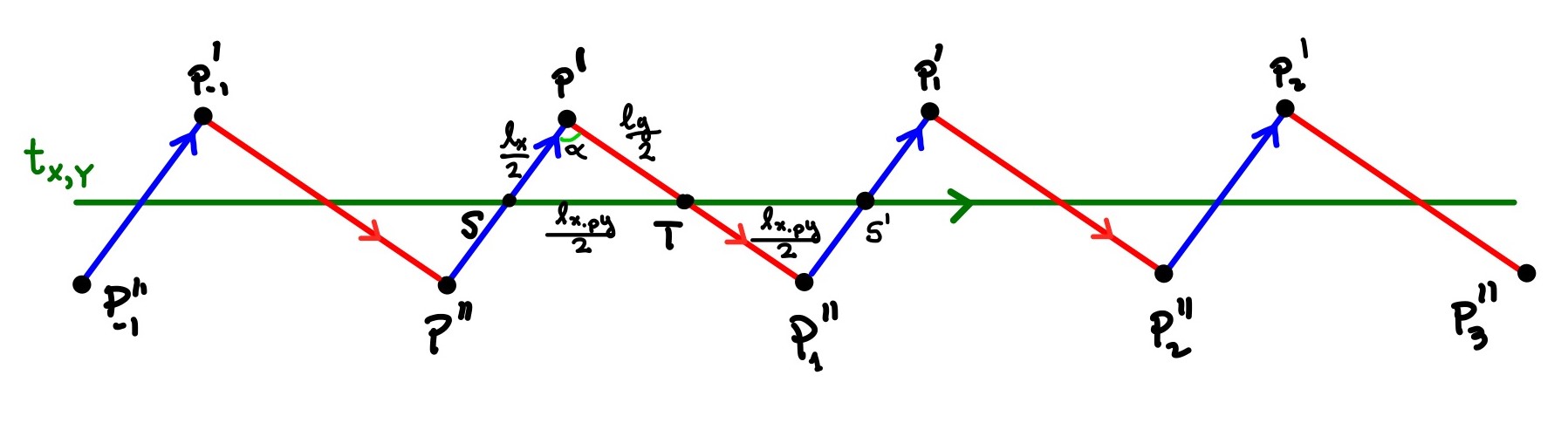}

\caption{Lengths of segments in Lemma~\ref{lem:beardon} }
	\label{fig:lengths}
\end{figure}

\subsection{The midpoint of a closed geodesic with respect to a point}

For each closed geodesic  $z$ in $\Sigma$ and each point $R$ in $z$, the \emph{midpoint of $z$ with respect to $R$, denoted by $R_z$} is the point of $z$ such that the lengths of both arcs of $z$ from $R$ to $R_z$ are equal. 

The next corollary follows from Lemma~\ref{lem:beardon}.

\begin{corollary}\label{cor:unique}   The unique closed directed geodesic  in the free homotopy class of $\Gp$ (which is also the unique geodesic in the class of the loop product of $x$ and $y$ based at $P$) is the projection (by the covering map) of the axis of the composition $X \circ Y$. Moreover, this closed  geodesic 
intersects $x$ in $P_x$, the midpoint of $x$ with respect to $P$, and $y$  in $P_y$ the midpoint of $y$ with respect to $P$. (See Figure~\ref{fig:unique})
\end{corollary}

\begin{remark}
In the notation of Corollary \ref{cor:unique}, in Figure \ref{fig:lengths}, $S$ projects to $P_x$ and $T$ projects to $P_y$.  

\end{remark}

\begin{figure}[ht]
	\centering
	\includegraphics[width=0.5\textwidth]{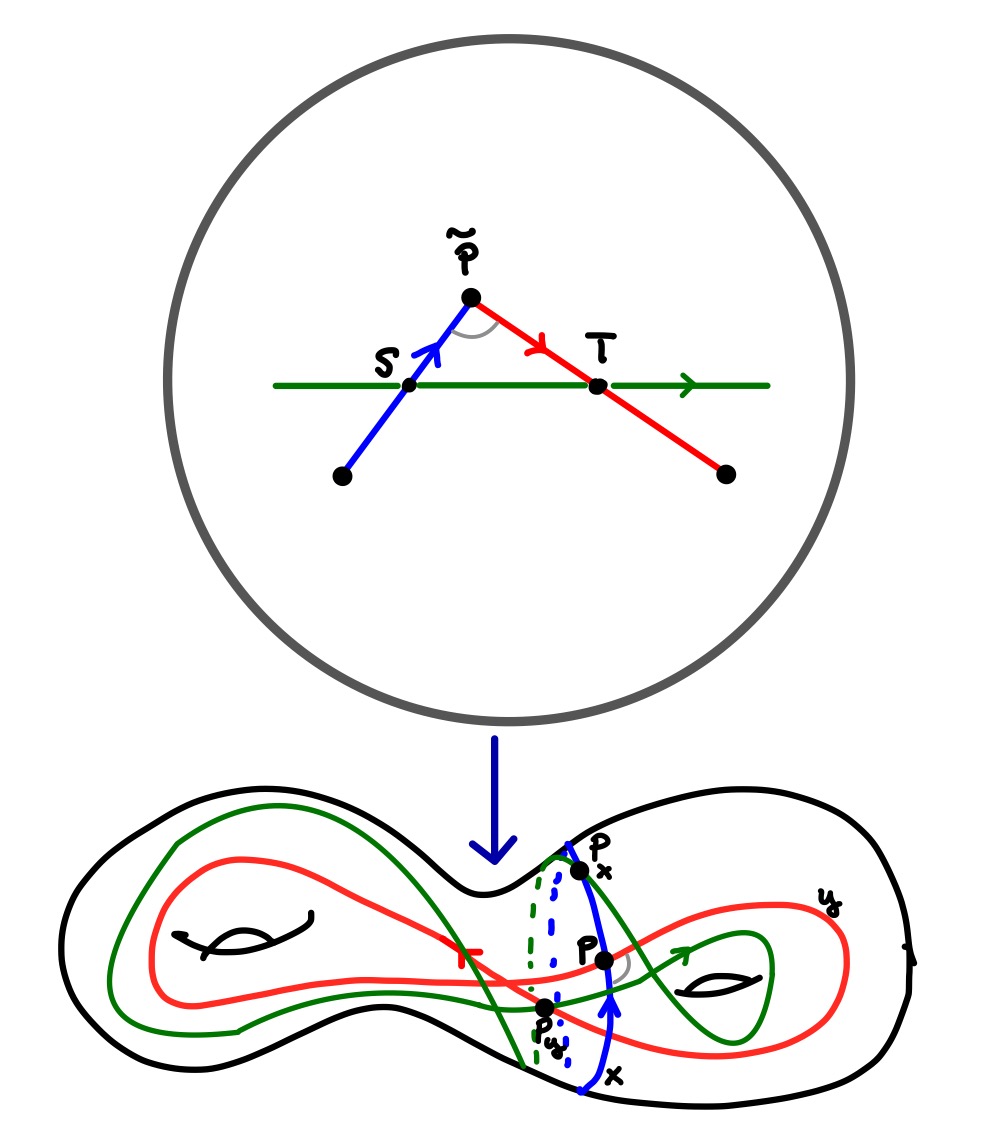}
	\caption{Product, geodesic and lift}
	\label{fig:unique}
\end{figure}

\section{Study of the term of the Goldman bracket associated with  a given intersection point of two closed directed geodesics}

\subsection{Symmetries of two cancelling terms}

Throughout this section, we will assume  that $P$ and $Q$ are two intersection  points of $x$ and $y$ such that  the corresponding terms of the Goldman bracket cancel. In particular, the closed, directed, piecewise geodesics $\Gp$ and  $\Gq$ are freely homotopic. Then there exists a lift $P'$ of $P$ and a lift $Q'$ of $Q$ such that the following happens. There is a full lift $\gq$ of $\Gq$ with the same endpoints as $\gp$ and the same endpoints as the line $\emm$. Since $\emm$ is also the unique geodesic in the class of  $\Gq$,  $\gq$  zigzags around $\emm$ analogously as  $\gp$. 

Since the terms corresponding to $P$ and $Q$ in the Goldman bracket cancel each other,  the signs associated with $P$ are $Q$ are distinct and the local picture  around $P$ and $Q$  is as in Figure~\ref{fig:signs}, possibly swapping $P$ and $Q$ (since the picture is local, the same ideas apply to the lifts $P'$ and $Q'$.)

\begin{figure}[ht]
	\centering
	\includegraphics[width=0.6\textwidth]{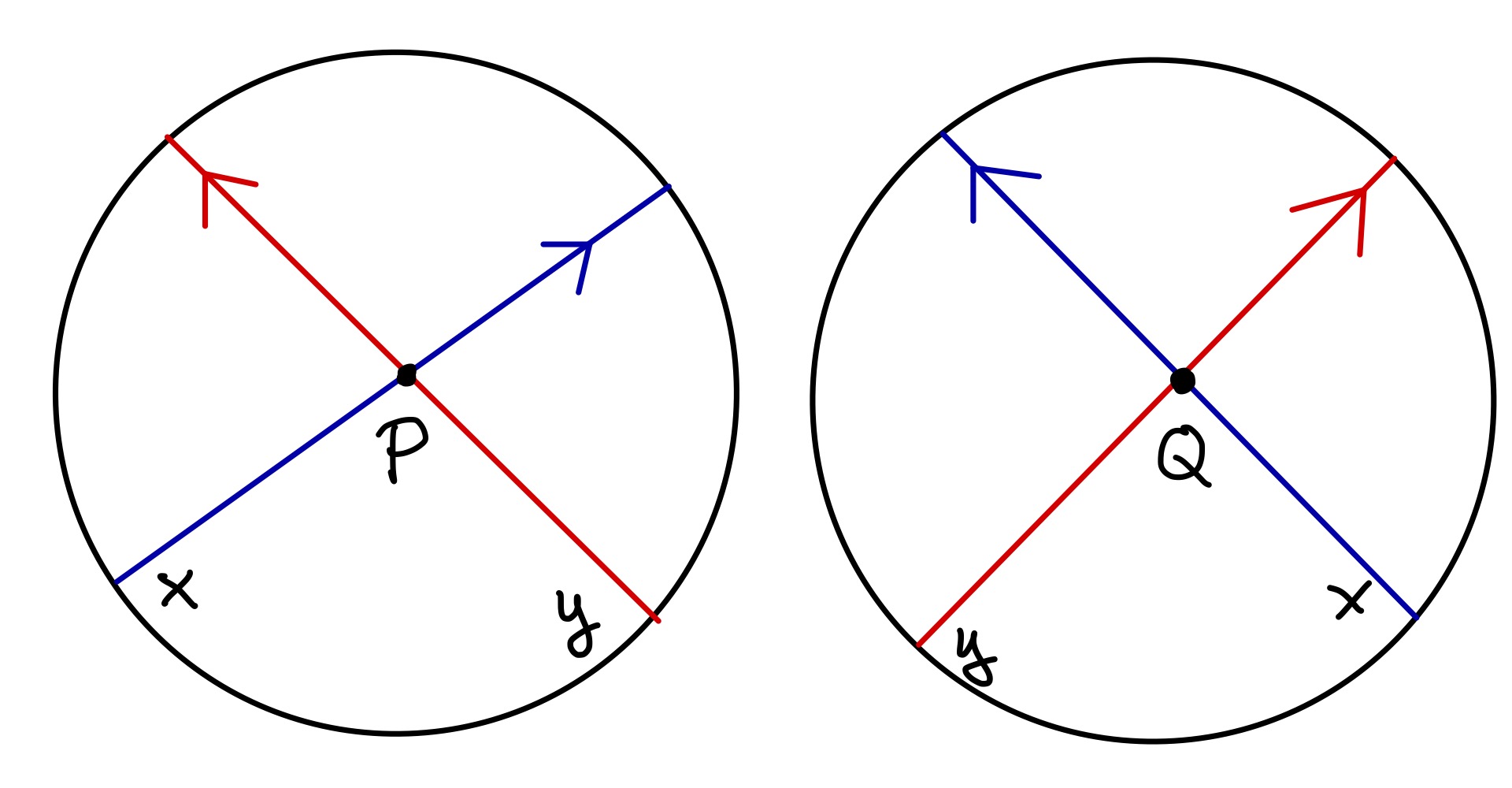}
	\caption{Signs determined by the  local picture around $P$ and $Q$.}
	\label{fig:signs}
\end{figure}

If $T$ is an intersection point of two directed lines (possibly a self-intersection point)  the \emph{forward angle at $T$}, denoted by $\overrightarrow{T}$  is the angle formed by the two forward directions of the lines.

The next corollary follows from Lemma~\ref{lem:beardon}.
\begin{corollary}{\label{cor:angle}}  The forward angles of at $P$ and $Q$  are congruent. In symbols, $\overrightarrow{P}=\overrightarrow{Q}.$
\end{corollary}

\begin{lemma}\label{lem:reflections} There exist two lines $\U$ and $\V$ perpendicular to the line  $\emm$ (the axis of $X \circ Y$) such that

\begin{enumerate}
\item The reflection about  $\U$ and the reflection about $\V$ swap $\gp$ and $\gq$.
\item The distance between  $\U$ and $\V$ is half of $\ell_{xy}$, the translation length of $X \circ Y$.
\item If $\ell_x = \ell_y$ one of the  lines, say $\U$ intersects an $x$-segment of $\gp$ and an $x$-segment of $\gq$(thus the lift of a self-intersection point of $x$) and the other  line, $\V$, intersects a $y$-segment of $\gp$ and a $y$-segment of $\gq$ (thus the lift of a self-intersection point of $y$). See Figure~\ref{fig:cases}, (2).
\item If $\ell_x \le \ell_y$, then there are two possibilities: one is as in the previous item (3), and the other is the following: 
Each of the  lines, $\U$  and $\V$ intersect an $x$-segment of $\gp$ and an $x$-segment of $\gq$(thus the lift of a self-intersection point of $x$). See Figure~\ref{fig:cases}, (4). 
\end{enumerate}
Moreover, all possible configurations of  $\U$, $\V$, $\gp$ and $\gq$ are those in Figure~\ref{fig:cases}. 
\end{lemma}
\begin{proof}

The idea of the proof is the following: the two piecewise geodesics $\gp$ and $\gq$ zigzag around the axis $\emm$. Since the corresponding terms have opposite signs, and the turning angles and corresponding pairs of segments are congruent, $\gp$ is the reflection of $\gq$ about $\U$,  a line perpendicular to the axis $\emm$. (In fact, because of the translation symmetries of the figure, there are infinitely many lines with that property). Also, both piecewise geodesics $\gp$ and $\gq$ are invariant under $X \circ Y$, of translation length $\ell_{xy}$. This translation can be written as the composition of a reflection about $\U$ and a reflection about $\V$ another line perpendicular to $t_{xy}$, at distance $\ell_{xy}$ of $\U$. Thus, the union of both lines is invariant under reflection about $\V$, which implies that $\V$ swaps $\gp$ and $\gq$.

\begin{figure}[ht]
	\centering
	\includegraphics[width=0.8\textwidth]{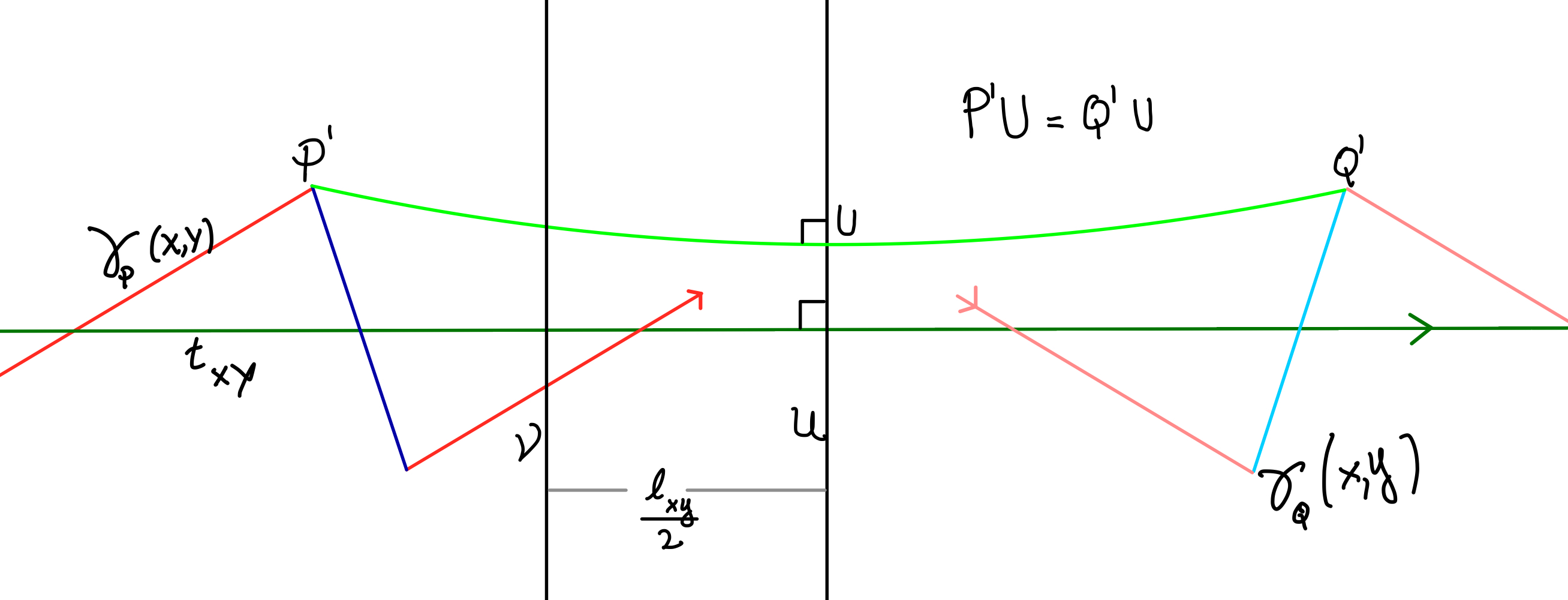}
\caption{Proof of Lemma~\ref{lem:reflections}.} The piecewise geodesics are not shown overlapping for clarity.
	\label{fig:symmetry}
\end{figure}
The formal proof is below.

Consider a preimage $Q'$ of the point $Q$ on $\gq$ which is on the same side as $P'$ with respect to  $\emm$, the axis of the transformation $X \circ Y$ .
Observe the following:
\begin{itemize}
\item The forward angles at  all the lifts of $P$ and $Q$ are congruent by Corollary~\ref{cor:angle}.
\item The length of pairs of corresponding $x$-segments (respectively $y$-segments)  forming the piecewise geodesics, drawn in blue and light blue  (respectively  red and light red) in Figure~\ref{fig:symmetry}, are  congruent. 
\item For each of the two piecewise geodesics, the length of the segment between two consecutive intersection points with the axis is equal to $\ell_{xy}/2$ (half the translation length of $X \circ Y$) and therefore they are  congruent (by Lemma~\ref{lem:beardon})
\end{itemize}

Let $U$ be the midpoint of the segment $P'Q'$ and let $\U$ be the perpendicular to $\emm$ passing through $U$. 
The   reflection $\rho_\U$ about $\U$ swaps $\gp$ and $\gq$.

On the other hand, by Lemma~\ref{lem:beardon}, the transformation  $X \circ Y$ acts as a translation of distance $\ell_{xy}$ on the line $\emm$. Consider now a line $\V$, also perpendicular to  $\emm$, at distance $\ell_{xy}/2$ from $\U$ and before $\U$ in the direction of  $X \circ Y$ . Then  $X \circ Y$  can be written as $\rho_{\U}\circ\rho_{\V}$. Since the union of the piecewise geodesics $\gp \cup \gq$ is invariant under  $\rho_{\U}$ and under $\rho_{\U}\circ\rho_{\V}$, then $\gp \cup \gq$ is also invariant under $\rho_{\V}$.

\begin{figure}[ht]
	\centering
	\includegraphics[width=0.4\textwidth]{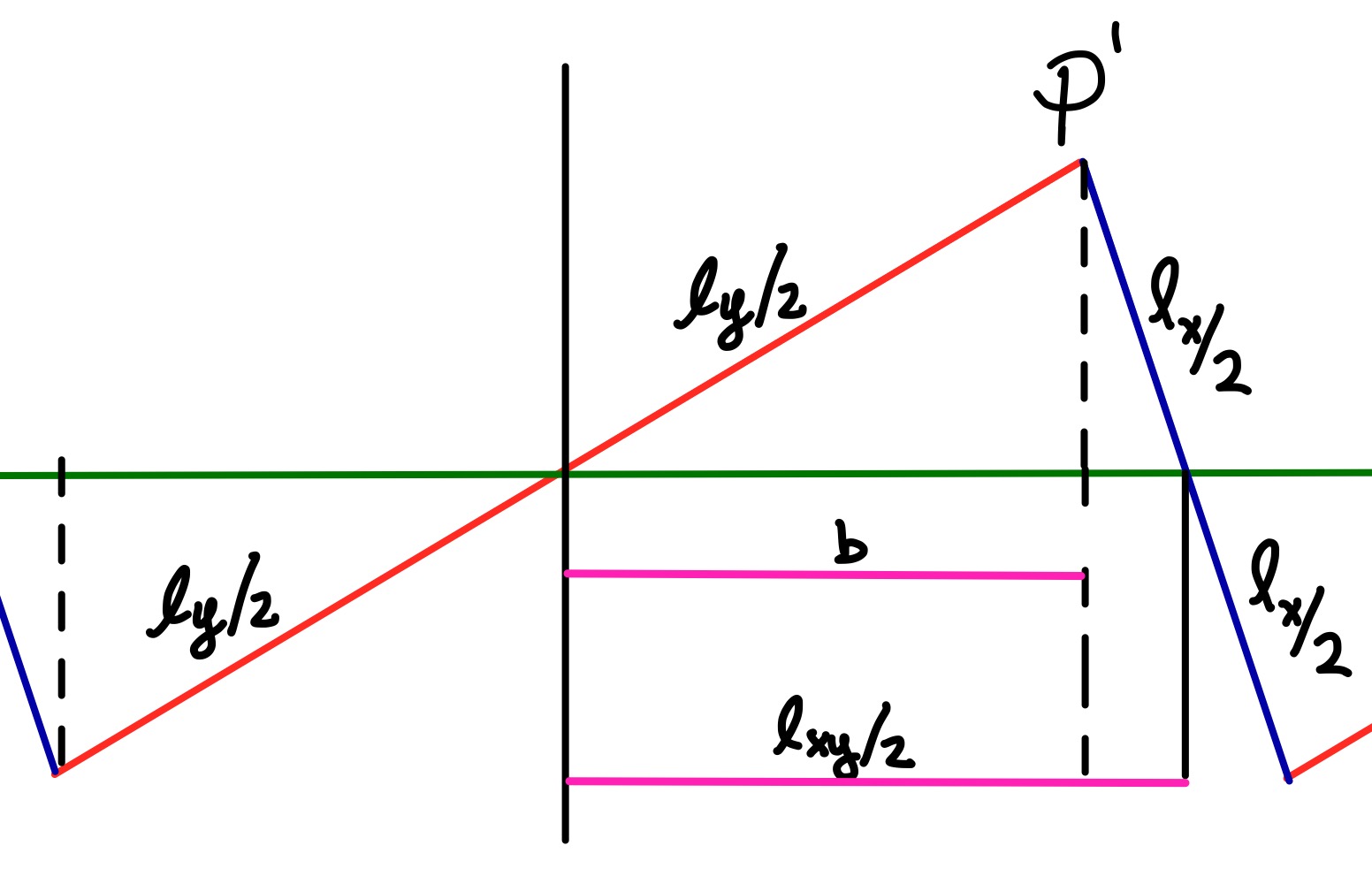}
\caption{Proof of Lemma~\ref{lem:reflections}}.
	\label{fig:b}
\end{figure}
The existence of self-intersection points on $x$ and $y$ follows from the following facts:

\begin{itemize}
\item $\U$ and $\V$ are at distance $\ell{xy}/2.$
\item the distance between the intersection of $\emm$ and  two consecutive laps of $x$ is $\ell_{xy}$.
\item If $\ell_y > \ell_x$, the  segment resulting from  the projection of a $y$-segment to the axis $\emm$ is smaller than $\ell_{xy}/2$. (See Figure~\ref{fig:b}, where the length of the projection of a $y$-segment is $2b$)
\item If $\ell_x = \ell_y$, the  segment resulting from  the projection of either an $x$-segment or a $y$-segment to the axis $\emm$ is  $\ell_{xy}/2$.
\end{itemize}

\begin{figure}[ht]
	\centering
	\includegraphics[width=1\textwidth]{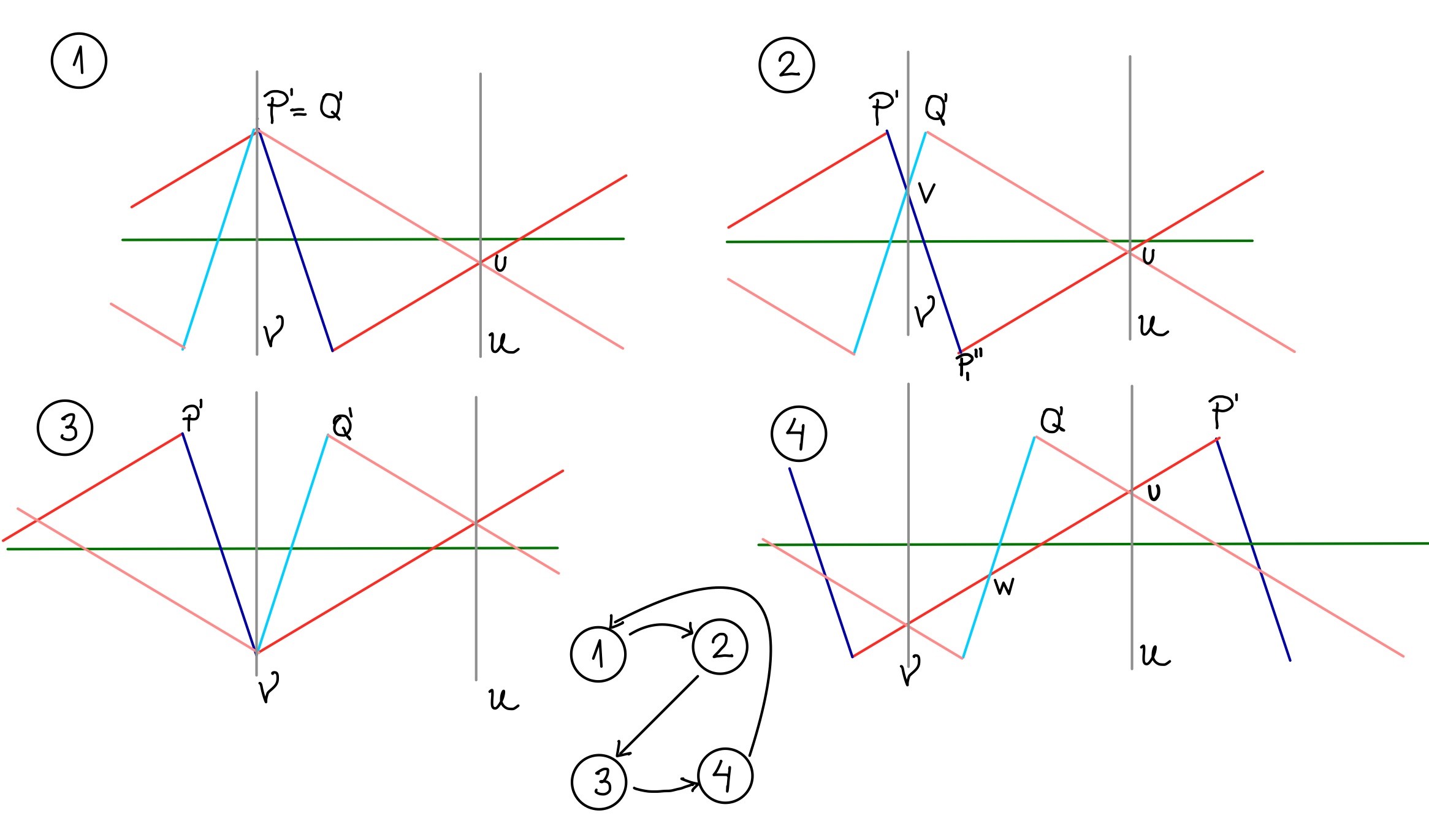}
\caption{Relative positions of $\U$, $\V$, $\gp$ and $\gq$}
	\label{fig:cases}
\end{figure}

All the possible configurations of $\gp$ and $\gq$ can be obtained  as follows: Choose any possible configuration of  $\gp, \gq$ and $\emm$ (In Figure~\ref{fig:cases} we started with the case (1) where $P'$ and $Q'$ coincide). Fix  $\gp$ and $\emm$ and then slide $\gq$ along $\emm$ in any direction up to length $\ell_{xy}$. (In our example, until $P'$ coincides again with a lift of $Q$).   
There are two generic  configurations, shown in Figure~\ref{fig:cases} (2) and (4): either $\U$ and $\V$ intersect the piecewise geodesics in the lift of a self intersection point of one of the geodesics (the longest one) or $\U$ and $\V$ intersect the piecewise geodesics in the lift of one self-intersection point of either of the geodesics.  (The extreme cases shown in Figure~\ref{fig:cases}(1) and (3), satisfy both descriptions). 
\end{proof}

\begin{lemma}\label{lem:lx>ly} If $\ell_x \le \ell_y$ then one of the following hold: 
\begin{enumerate}
    \item There exist an intersection point $W$ of $x$ and $y$
    such that $\overrightarrow{W} < \overrightarrow{P}$ (Figure~\ref{fig:cases}, (4)).
    
    \item There exist a self-intersection point $V$ of $x$ and a self-intersection point $U$ of $y$  such that  either
    $\overrightarrow{U} < \overrightarrow{P}$, or 
    $\overrightarrow{V} < \overrightarrow{P}$. (Figure~\ref{fig:cases}, (2))
\end{enumerate}

\end{lemma}
\begin{proof}
    By Lemma~\ref{lem:reflections} (and with the same notations) all the configurations  of $\U$, $\V$, $\gp$ and $\gq$ are as in Figure~\ref{fig:cases}.
    
	Suppose that the relative  positions of $\U$, $\V$, $\gp$ and $\gq$ are as in Figure~\ref{fig:cases}(4). The sector of the plane limited by lines $\U$ and $\V$, contains a full lap of a lift of $x$ (in light blue Figure~\ref{fig:cases}) and is traversed by part of a lap of $y$ (in red). As the figure shows, the $x$-segment and the $y$-segment  intersect at a point $W$. 

    By the Exterior Angle Theorem, the exterior angle at $Q'$ of the triangle $UQ'W$ (that is, the directed angle $\overrightarrow{Q'}$) is larger than the interior angle at $W$ (that is, the directed angle $\overrightarrow{W}$). In symbols, $\overrightarrow{W}< \overrightarrow{Q'}$. Thus, (1) holds.

    If the relative  positions of $\U$, $\V$, $\gp$ and $\gq$ are as in Figure~\ref{fig:cases}(2), a quadrilateral $VQ'UP_1''$ is determined. The desired result (2) follows from the fact that the sum of the interior angles of a quadrilateral is less than $2\pi$.\end{proof}

\begin{lemma}\label{lem:lx=ly}  If $\ell_x = \ell_y$ then there exist a self-intersection point $U$ of $x$  and a self-intersection point $V$ of $y$ such that $\overrightarrow{U}=\overrightarrow{V} < \overrightarrow{P}=\overrightarrow{Q}.$ (Figure~\ref{fig:lx=ly}.) 
\end{lemma}
\begin{proof}
\begin{figure}[ht]
	\centering
	\includegraphics[width=.5\textwidth]{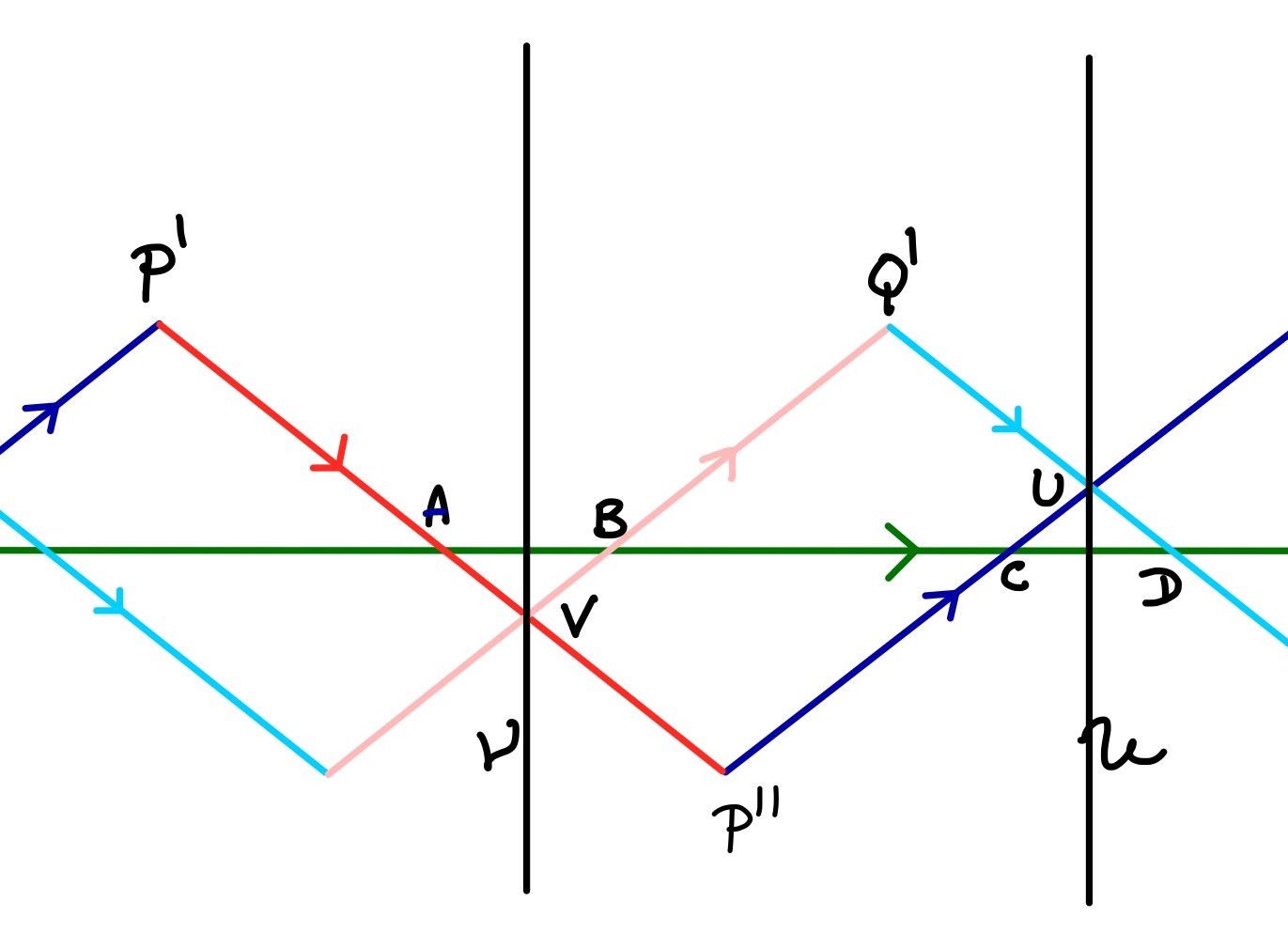}
\caption{Study of case $\ell_x=\ell_y$}.
	\label{fig:lx=ly}
\end{figure}

We refer to Figure~\ref{fig:lx=ly} for the labels of the geometric objects. 

Note that the segments $AC$ and $BD$ are congruent since their length is $\ell_{xy}/2$. Hence $AB=AC-BC=BD-BC=CD$. The triangles $VAB$ and $UCD$ are congruent by angle-side-angle. This implies that $\overrightarrow{U}=\overrightarrow{V}$. 

The triangles $BQ'D$ and $AP''C$ are congruent by SAS (side-angle-side). Then the interior angle at $B$ and the interior angle at $C$ add up to $\pi$.

Since the sum of the interior angles of the quadrilaterlal UQ'BC is less than $2\pi$, and the interior angle at $B$ and the interior angle at $C$ add up to $\pi$, $\pi-\overrightarrow{Q'}+\overrightarrow{U}<\pi$, as desired.
 
\end{proof}

\subsection{Characterization of simple curves via powers }

\begin{theorem}\label{theo:simple}
If $[x,x^n]=0$ for some $n \ge 2$ then $x$ is simple. 
\end{theorem}
\begin{proof} Suppose that $x$ is not simple and that $[x,x^n]=0$. Let $P$ and $Q$ be two self-intersection points of $x$, such that the corresponding terms cancel and that the forward angles at $P$ (and $Q$) are the smallest among all forward angles of self-intersection points of $x$. By Lemma~\ref{lem:lx>ly}, there is a self-intersection point of $x$, with a forward angle smaller than $\overrightarrow{P}$. Since $[x,x^n]=0$, the term corresponding to this self-intersection point cancels, which is a contradiction. Thus, the corollary holds. 
\end{proof}


\subsection{Characterization of simple curves via $x,\bar{x}$ }

Recall that we denote by $\bar{x}$  the curve $x$ with opposite direction.

\begin{theorem}\label{theo:xbarx}
 Let $x$ be a non-power geodesic. Then $[x,\bar{x}]=0$ if and only if $x$ is simple. 
\end{theorem}
\begin{proof}
If $x$ is simple then by definition of the bracket $[x,\bar{x}]=0$. 

To study the other implication, we argue by contradiction. Suppose that there exists a non-power geodesic $x$, such that  $[x,\bar{x}]=0$ and  $x$ is not simple. Then there are pairs of self-intersection points of $x$ that cancel. 
Let $P$ and $Q$ one of such pairs, with the property that   the forward angles at $P$ and $Q$ are the largest among all forward angles of self-intersection points of $x$.

We will show that there exists another self-intersection point of $x$ whose forward angle is larger than that of $P$.

By Lemma~\ref{lem:lx=ly}  the  configuration of $\gp$ and $\gq$ is as in the Figure \ref{fig:lx=ly}.

\begin{figure}[h]
    \includegraphics[width=0.8\textwidth]{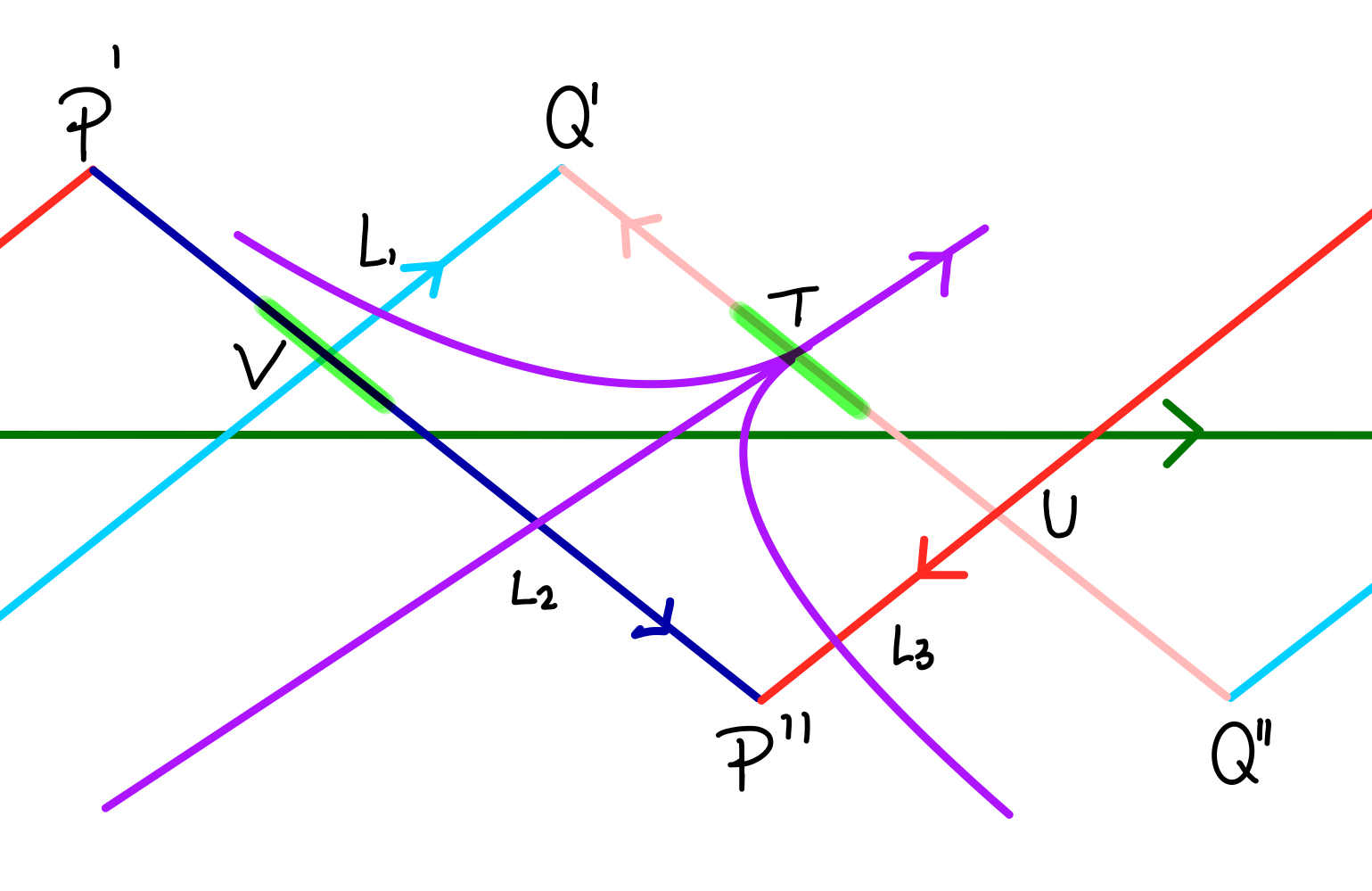}

    \caption{Proof of Theorem~\ref{theo:simple}. The arrows in the red and blue geodesic segments denote the direction of $x$}\label{fig:xbarx}
\end{figure}

Consider a small segment $s$ of $P'P''$ on a  neighborhood of  $V$ (highlighted in green in Figure~\ref{fig:xbarx}).  Since $[Q'Q'')$ is full lift of $x$ (its length being $\ell_x$), $[Q'Q'')$ contains a lift of the projection of the segment $s$, and a point $T$, a lift of the projection of $V$.
Hence, there is an orientation preserving isometry  mapping $V$ to  $T$  and the interval $s$ to an interval (or union of two intervals in the extreme case $T=Q'$) in $[Q',Q'')$. 

There are four possibilities for the point $T$: 
$T=Q'$,
$T$ in $(Q',U)$,
  $T=U$, or
$T$ in $(U,Q')$.

If $T=Q'$ then there is an isometry mapping a small segment of $P''P'$ starting at $V$ and going in direction to $P'$to a small segment starting at $Q'$ and going in direction of $Q''$. This implies the $x$ is a proper power of another geodesic, contradicting our hypothesis. A similar argument shows that the assumption $T=U$ leads to a contradiction.


We are going to study now the case $T$ in $(Q',T)$ (the case $T$ in $(T,Q'')$ follows similarly). Since  $T$ is an intersection there is another lift of $x$ through $T$, such that $\overrightarrow{T}=\overrightarrow{V}$. This branch has to enter the parallelogram $Q'VP''U$ though one of the following segments: $Q'V$, $VP''$ or $P''U$, see Figure~\ref{fig:xbarx}. 

If the lift of $x$ through $T$ enters the parallelogram through $Q'V$, we have a triangle $L_1 TQ'$ as in Figure~\ref{fig:xbarx}. By the Exterior Angle Theorem, $\overrightarrow{Q'}<\overrightarrow{T}$. This contradicts the assumption that the angle at $Q'$ is the largest.

If the lift of $x$ through $T$ enters the parallelogram through $VP''$, we have a paralelogram $L_2TUP''$. Since 
$\overrightarrow{T}=\overrightarrow{V}$,  By Lemma~\ref{lem:lx=ly}, $\overrightarrow{U}=\overrightarrow{V}$. Since the sum of the interior angles of $L_2TUP''$ is less than $2\pi$, we have that $\overrightarrow{Q}<\overrightarrow{L_2}$ once more contradicting  the assumption that the angle at $Q'$ is the largest.

If the lift of $x$ through $T$ enters the parallelogram through $P''U$, then $\overrightarrow{U}<\overrightarrow{T}$ by the Exterior Angle Theorem. This contradicts the fact that $\overrightarrow{U}=\overrightarrow{T}$. Thus, the proof is complete. \end{proof}

 \bibliographystyle{plain}

\bibliography{gla.bib}
\end{document}